\documentclass[12pt]{article}
\usepackage{amssymb,amsmath, amsthm}
\usepackage{color,xcolor}
\newcommand{\eb}{\begin{equation}}
\newcommand{\ee}{\end{equation}}
\newcommand{\ebx}{\begin{equation*}}
\newcommand{\eex}{\end{equation*}}
\newtheorem{lemma}{Lemma}[section]
\newtheorem{proposition}[lemma]{Proposition}
\newtheorem{theorem}[lemma]{Theorem}

\newtheorem{corollary}[lemma]{Corollary}
\newtheorem{definition}[lemma]{Definition}

\allowdisplaybreaks

\makeatletter
\renewcommand*\env@matrix[1][*\c@MaxMatrixCols c]{%
  \hskip -\arraycolsep
  \let\@ifnextchar\new@ifnextchar
  \array{#1}}
\makeatother

\begin{document}

\title{Rigidity of holomorphic maps between Shilov boundaries of type-I bounded symmetric domains}

\author{Yun Gao\footnote{School of Mathematical Sciences, Shanghai Jiao Tong University, Shanghai, People's Republic of China. \textbf{Email:}~gaoyunmath@sjtu.edu.cn}}

\maketitle

\begin{abstract}
	
	We introduce  the concept  of orthogonal structure on complex Grassmannians. Based on this structure, we define the notion of orthogonal mappings. This class of maps generalizes holomorphic maps between the Shilov boundaries of type-I bounded symmetric domains. By analyzing the geometric properties of these orthogonal mappings, we obtain rigidity results for such mappings. As an application, we establish a rigidity result for the holomorphic maps from the Shilov boundary of rank $1$ type I bounded symmetric domain $\Omega_{1,s}$ (i.e. unit spheres) to the Shilov boundary of a higher-rank type-I domain $
	\Omega_{r',s'}$, where $s\ge 2$ and $2\le r'\le s'$. More specifically, we show that: (1) such maps are constant when $s'-r' < s-1$; (2) for $s-1 \le s'-r' < 2s-2$, they reduce to the standard linear embeddings after normalization by applying automorphisms on both sides. Our results are optimal and  generalize  the well-known optimal bounds for  proper mappings between rank $1$ cases and CR maps between Shilov boundaries of higher-rank type I bounded symmetric domains. 
\end{abstract}

\section{Introduction}

Rigidity of holomorphic/CR maps between bounded symmetric domains is a classical topic, extensively studied since Bochner and Calabi. While bounded symmetric domains exhibit rich rigidity phenomena, understanding differs significantly by rank. For an extensive literature, the reader is referred to Mok  (\cite{Mok1}), as well as many references therein.

The rigidity for the proper holomorphic maps between balls (rank 1 bounded symmetric domains) has received considerable attention with its origins tracing back to Poincar\'{e}'s work for $2$-dimensional unit ball. For unit balls in different dimensions, this field has been extensively explored by many mathematicians from various directions including algebraic geometry, Cauchy-Riemann geometry, and partial differential equation, etc. see for example  Cima-Suffridge \cite{CS}, D'Angelo \cite{Da1}, \cite{Da2}, D'Angelo-Kos-Riehl \cite{DKR}, D'Angelo-Lebl \cite{DL1},\cite{DL2}, Faran \cite{Fa}, Forstneric \cite{F1},\cite{F2},  Huang \cite{Hu1},\cite{Hu2}, Huang-Ji \cite{HJ}, and Huang-Ji-Yin \cite{HJY}, among others.

Compared to the rank 1 case, rigidity properties for holomorphic maps between higher rank bounded symmetric domains are much less understood.
Tsai's theorem showed if $ f: \Omega \to \Omega' $ is a proper holomorphic map between irreducible bounded symmetric domains with $\text{rank}(\Omega) \geq \text{rank}(\Omega') \geq 2 $, then $\text{rank}(\Omega) = \text{rank}(\Omega') $ and $f$ is a totally geodesic isometric embedding. This result resolved a conjecture of Mok \cite{Mok2} and highlighted the importance of rank conditions in determining the rigidity of such maps. For more results about the proper holomorphic maps with  $ \text{rank}(\Omega)< \text{rank}(\Omega')  $, we refer the readers to Chan\cite{Ch1},\cite{Ch2},
Henkin-Novikov\cite{HN},  Kim\cite{Kim}, Kim-Mok-Seo \cite{KMS}, Kim-Zaitsev\cite{KZ13, KZ15}, Mok \cite{Mok2}, Mok-Ng-Tu\cite{MNT}, Ng[Ng1-3], Seo[Seo1-3],  Tu[Tu1-2] to name a few. Almost all the aforementioned results are for $\text{rank}(\Omega)\ge 2$.
In \cite{XY}, the authors proved a rigidity result for proper holomorphic maps from the unit ball (rank $1$) to the Type IV classical domain (rank $2$) when the codimension is small. On the other hand, the case of rigidity for holomorphic maps from rank $ r = 1 $ to higher-rank Type I bounded symmetric domains remains unresolved due to limitations in the analysis of boundary components.

Let $r\le s \in \mathbb N$. Type I bounded symmetric domains \(\Omega_{r,s} = \{ Z \in \mathbb{C}^{r \times s} \mid I_r - ZZ^H > 0 \}\) with rank \(r\), along with their Shilov boundaries \(S(\Omega_{r,s}) = \{ Z \in \mathbb{C}^{r \times s} \mid I_r - ZZ^H = 0 \}\) where $Z^H=\bar Z^t$, have been studied by Kim-Zaitsev in \cite{KZ13} with regard to CR rigidity.  
The Shilov boundary imposes stringent constraints on holomorphic maps, forcing them to align with canonical substructures. Kim-Zaitsev show that CR mappings between $S(\Omega_{r,s})$ and $S(\Omega_{r',s'})$ are linear under the assumptions $s'-r'<2(s-r)$, $2\le r<s$, $2\le r'<s'$, and certain smoothness condition, using the Cartan moving frame method. In a very recent paper \cite{Kim2}, the author also describes explicitly the holomorphic mappings between Type I bounded symmetric domains that preserve Shilov boundary.  These results crucially depend on the domains having $r$, $r'\ge 2$.  The analog of Kim-Zaitsev's theorem in the rank $1$ case was solved earlier by Faran \cite{Fa}.

Most of the known rigidity results about proper holomorphic mappings between balls rely heavily on Tanaka-Chern-Moser theory \cite{CM}, \cite{Ta}, which is unavailable for higher rank bounded symmetric domains. On the other hand, the Cartan moving frame method does not work effectively for rank $1$ bounded symmetric domains. Hence, the case of mappings from rank $1$ domains (complex balls) to higher-rank domains requires a different approach.

We introduce a novel geometric approach to this problem by defining an orthogonal structure on Grassmannians and studying associated orthogonal mappings between them. This framework, integrating orthogonality with an effective analysis of linear subspaces and projective geometry techniques, yields rigidity results without resorting to heavy machinery. These orthogonal mappings slightly generalize the holomorphic maps between the Shilov boundaries of Type I bounded symmetric domains. Consequently, our approach provides new rigidity results for such maps. Here, by a holomorphic map between the Shilov boundaries, we mean a holomorphic map $F$ defined on a neighborhood $U$ of a point on $S(\Omega_{r,s})$ such that $F(U\cap S(\Omega_{r,s}))\subset S(\Omega_{r',s'})$.

Let $r\le s \in \mathbb N$ and denoted by $\mathbb C^{r,s}$ the Euclidean space $\mathbb C^{r+s}$ with the indefinite inner product
$$
\langle z, w\rangle_{r,s}
=z_1\bar w_1+\cdots+z_r\bar w_r-z_{r+1}\bar w_{r+1}-\cdots - z_{r+s}\bar w_{r+s},
$$
where $z=(z_1,\ldots,z_{r+s})$ and $w=(w_1,\ldots,w_{r+s})$. Denote $\mathbb P\mathbb C^{r,s}=\mathbb P^{r,s}$.

A positive $r$-plane (resp. null $r$-plane) in $\mathbb C^{r,s}$ is a linear subspace where $\langle\cdot,\cdot\rangle_{r,s}$ restricts to a positive-definite form (resp. a null form). The Type I domain $\Omega_{r,s}$ parametrizes such positive $r$-planes, while its Shilov boundary $S(\Omega_{r,s})$ corresponds to null $r$-planes in $\mathbb C^{r,s}$.

Consider the Grassmannian variety $G(r,r+s)$, whose points bijectively correspond to the $r$-dimensional linear subspaces in
$\mathbb C^{r+s}$. Let $p\in G(r,r+s) $. We use $V_p$ to denote the  $r$-dimension linear subspace corresponding to $p$. An orthogonal structure can be defined on $G(r,r+s)$ induced by the inner product $\langle\cdot,\cdot \rangle_{r,s}$. We call two points $p$, $q$ orthogonal in $G(r,r+s)$ if and only if $V_p\perp V_q$ in $\mathbb C^{r,s}$, denoted by $p\perp q$. We then denote by $\mathcal {G}(r,s)$ the Grassmannian variety $G(r,r+s)$
equipped with the orthogonal structure. This leads to the definition of $\textit{local orthogonal maps}$.
\begin{definition}\label{orth}
	Let $ U \subseteq \mathcal{G}(r,s) $ be a connected open set containing a null point. A holomorphic map $ F: U \to \mathcal{G}(r',s') $ is called \emph{orthogonal} if $ F(p) \perp F(q) $ for all $ p, q \in U $ such that $p\perp q$. We also simply call $F$ a \emph{local orthogonal map} from $ \mathcal{G}(r,s) $ to $ \mathcal{G}(r',s') $.
\end{definition}

In our definition of orthogonal maps, we required the domain U to contain a null
point, ensuring the orthogonality condition would not be vacuous. However, this
condition is not essential to the definition. Orthogonal maps can be defined more
generally, requiring only that $U$ contains at least one pair of orthogonal points. From
Proposition \ref{orthShilov}, we shall see that the local orthogonal maps from $\mathcal G(r,s)$ to $\mathcal G(r',s')$ and the local holomorphic maps between Shilov boundary $S(\Omega_{r,s})$ and $S(\Omega_{r',s'})$ are actually almost the same set of mappings although the former  could be defined slightly more generally. 

We are now ready to state our main theorems. Further explanation and details can be found in Section 3. 

\begin{theorem}\label{main}
	Let $s\ge 2$, $2\le r'\le s'\in \mathbb N$ and let $F$ be a local orthogonal map from $ \mathcal G(1,s)$ to $\mathcal G(r',s')$. 
	\begin{itemize}
	\item[(1)]If $s'-r'<s-1$, then $F$ is a constant map.
	
	\item[(2)] If $s-1\le s'-r'< 2s-2$, then after composing with suitable automorphisms of $\mathcal G(1,s)$ and $\mathcal G(r',s')$,
		$F$ is given by the following
		$$[z_0,\bf z]\to \left(\begin{array}{ccccc}
			I&0&I&0&0\\
			0&z_0&0&\bf z &0\\
		\end{array}\right),$$
		where $[z_0,{\bf z}]=[z_0,z_1,\cdots,z_{s}]\in \mathcal G(1,s)\cong \mathbb P^{1,s}$.

	\end{itemize}

\end{theorem}

Furthermore, the inequalities governing the rigidity are sharp, as evidenced by the generalized Whitney map. Beyond these bounds, rigidity fails.

As a corollary, we have

\begin{corollary}\label{main2}
	Let $s\ge 2$, $2\le r'\le s'\in \mathbb N$, and let $U\subset \mathbb C^{s}$ be an open neighborhood of a point in  $ S(\Omega_{1,s})$.
	let $F$ be a holomorphic  map from $U$ to $\mathbb C^{r's'}$ satisfying 
	$F(U\cap S(\Omega_{1,s}) )\subset S(\Omega_{r',s'})$.
	\begin{itemize}
	\item[(1)]If $s'-r'<s-1$, then $F$ is a constant map.
	
	\item[(2)] If $s-1\le s'-r'< 2s-2$, then after composing with suitable automorphisms of $ \Omega_{1,s}$ and $ \Omega_{r',s'}$,
		$F$ is given by the following
		$$z\to \left(\begin{array}{ccc}
			I&0&0\\
			0&  z &0\\
		\end{array}\right).$$

	\end{itemize}

\end{corollary}

We note that our proof of Theorem \ref{main} is centered around the orthogonality preservation property and the dimensional analysis deduced from it. This new perspective enables us to avoid sophisticated computations as in the previous works, and makes our arguments considerably shorter.

\section{Orthogonal structure on Grassmannian}

In \cite{GN}, the authors introduce an orthogonal structure in projective space to solve the rigidity problems for proper holomorphic mappings among generalized balls. Building upon this framework, in the current section we will generalize the definition of orthogonal structure to the Grassmannian manifold.

Let $r,s,t\in\mathbb N$ and $n:=r+s+t>0$. Define the indefinite inner product of signature $(r;s;t)$ on $\mathbb C^{n}$ as:
$$
	\langle z, w\rangle_{r,s,t}
	=z_1\bar w_1+\cdots+z_r\bar w_r-z_{r+1}\bar w_{r+1}-\cdots - z_{r+s}\bar w_{r+s},
$$
where $z=(z_1,\ldots,z_{n})$ and $w=(w_1,\ldots,w_{n})$. We also define the indefinite norm  $\|z\|^2_{r,s,t}=\langle z, z\rangle_{r,s,t}$. Then, for any $z\in \mathbb C^{r,s,t}$ endowed with this inner product, we call it a \textit{positive point} if $\|z\|^2_{r,s,t}>0$; a \textit{negative point} if  $\|z\|^2_{r,s,t}<0$ and a \textit{null point} if  $\|z\|^2_{r,s,t}=0$. Two vectors $z$, $w$ are called orthogonal (denote $z\perp w$) if $\langle z, w\rangle_{r,s,t}=0$. The \textit{orthogonal complement} of $z$ is defined as
$$z^{\perp}=\{w\in  \mathbb C^{r,s,t} \mid \langle z, w\rangle_{r,s,t}=0\}.$$

We denote by $\mathbb C^{r,s,t}$ the $\mathbb C^{n}$ with the Hermitian inner product defined above and $\mathbb P^{r,s,t}:=\mathbb P\mathbb C^{r,s,t}$. We write $\mathbb C^{r,s}$ and $\mathbb P^{r,s}$ instead of $\mathbb C^{r,s,0}$ and $\mathbb P^{r,s,0}$.

Consider a complex linear subspace $V\subset \mathbb C^{r,s,t}$ such that the restriction of  $\langle\cdot,\cdot\rangle_{r,s}$ on $V$ has the signature $(a; b; c)$;  we refer to $ V$ as an \textit{ $(a,b,c)$-subspace} of $\mathbb C^{r,s,t}$. 

If $a=b=0$, $V$ is called a \textit{null space}. Similarly, it is called a \textit{positive space} (resp. \textit{negative space}) if $b=c=0$ (resp. $a=c=0$).  We will also use the terms \textit{null $k$-plane, positive $k$-plane}, and \textit{negative $k$-plane} when $\dim V=k$.
Obviously, the maximum dimension of null spaces in $\mathbb P^{r,s}$ is $\min\{r,s\}$.  Null spaces with the maximal dimension are called the \textit{maximal null spaces}.

Let $r$, $s$ be positive integers and $r\le s$. Consider the Grassmannian variety $G(r,r+s)$, whose closed points are in bijection with  
$r$-dimensional vector subspaces 
 $V\subset \mathbb{C}^{r+s}$. Fix the standard basis $(e_1, \ldots ,e_{r+s})$ of $\mathbb{C}^{r+s}$. A linear subspace 
 $V$ can be represented as the row space of an $r \times (r+s)$ matrix:

$$A=\left(\begin{array}{cccccc}
	a_{11}&a_{12}&\cdots&a_{1,r+s}\\
a_{21}&a_{22}&\cdots&a_{2,r+s}\\
	\vdots&\vdots&\ddots&\vdots\\
			a_{r1}&a_{r2}&\cdots&a_{r,r+s}\\
			\end{array}\right),
$$
where the rows of $A$ encode a set of spanning vectors of $V$ with respect to the standard basis.
Such $A$ is called a \textit{representative matrix} of $V$.

The matrix $A$ is not unique: left multiplication by any invertible $r\times r$ matrix $X$ yields another representative matrix $XA$ for the same subspace $V$.
 Conversely, if two matrices $A$ and $A'$ represent the same subspace, there exists an invertible $r\times r$ matrix 
$X $ such that $A'=XA$.

 {\bf Notations:} For a point $p\in G(r,r+s)$, let $V_p\subset \mathbb{C}^{r+s}$ denote the $r$-dimensional linear subspace 
 corresponding to $p$, and let $A_p$ denote a representative matrix of $V_p$. If the row vectors $\alpha_1,\cdots,\alpha_r$ of $A_p$ are pairwise orthogonal with respect to the indefinite Hermitian inner product on $\mathbb C^{r,s}$ (i.e. $\alpha_i \perp \alpha_j$ in $\mathbb C^{r,s}$ ), then $A_p=(\alpha_1,\cdots,\alpha_r)^t$ is called a representative matrix of $p$,  an \textit{orthogonal representative matrix} for $p$.

 By equipping $\mathbb C^{r,s}$ with an indefinite inner product, we naturally induce an orthogonal structure on the Grassmannian $G(r,r+s)$. This structure categorizes subspaces based on their metric properties:

\begin{definition}

Let two points $ p, q \in G(r,r+s) $ be represented by $r\times (r+s)$
 matrices $A_p$ and $A_q$. With the signature matrix $$ I_{r,s} = \begin{pmatrix} I_r & 0 \\ 0 & -I_s \end{pmatrix}, $$ 
 we say $p$ is orthogonal to $q$(denoted $p\perp q$) if and only if
\[
A_pI_{r,s} A_q^H = 0,
\]
where $ A^H=\bar A^t $ denotes the Hermitian transpose.

A point $ p \in \mathcal{G}(r,s) $ is classified as
\begin{itemize}
\item[(1)] \emph{null} if
$
A_p I_{r,s} A_p^H = 0 \quad 
$
\item[(2)] \emph{positive} if $
A_p I_{r,s} A_p^H > 0 \quad 
$
\end{itemize}

Here, $\mathcal{G}(r,s)$ denotes the Grassmannian $ G(r,r+s) $
 endowed with this orthogonal structure.

\end{definition}

Obviously, for any two points $p$,$q\in \mathcal{G}(r,s)$,
$$p\perp q \iff V_p \subseteq (V_q)^\perp \subseteq \mathbb{C}^{r,s},$$ where $ (V_q)^\perp $ denotes the orthogonal complement with respect to $ \langle \cdot, \cdot \rangle_{r,s} $.

Hence,
if $ p \perp q $ in $ \mathcal{G}(r,s) $, then every vector $ \alpha \in V_p $ is orthogonal to every vector $ \beta \in V_q $ with respect to $ \langle \cdot, \cdot \rangle_{r,s}$. Furthermore, we immediately observe that $ \mathcal{G}(1,s) $ coincides with the projective space $ \mathbb{P}^{1,s} $.

Let $ M_{r,s} $ denote the space of $ r \times s $ complex matrices. The \emph{type-I irreducible bounded symmetric domain} $ \Omega_{r,s} \subseteq M_{r,s} \cong \mathbb{C}^{r \times s} $ is defined by
\[
\Omega_{r,s} = \{ A \in M_{r,s} \mid I_r - AA^H > 0 \}.
\]
Here we denote by $>0$ the positive definiteness of square matrices.
 This domain corresponds bijectively to the set of positive $ r $-planes in $ \mathbb{C}^{r,s} $, making $ \Omega_{r,s} $ precisely the set of positive points in $ \mathcal{G}(r,s) $. The \emph{Shilov boundary} of $ \Omega_{r,s} $
\[
S(\Omega_{r,s}) = \{ A \in M_{r,s} \mid I_r = AA^H \}
\]
coincides with the set of null points in $ \mathcal{G}(r,s) $. Moreover, if $ p \in S(\Omega_{r,s}) $, then $ V_p $ is a maximal null subspace of $ \mathbb{C}^{r,s} $.

\begin{definition}\label{orth}
Let $ U \subseteq \mathcal{G}(r,s) $ be a connected open set containing a null point. A holomorphic map $ F: U \to \mathcal{G}(r',s') $ is called \emph{orthogonal} if $ F(p) \perp F(q) $ for all $ p, q \in U $ such that $p\perp q$. We also simply call $F$ a \emph{local orthogonal map} from $ \mathcal{G}(r,s) $ to $ \mathcal{G}(r',s') $.
\end{definition}

{\bf Remark:} For the sake of simplicity, in our definition we require the domain $U$ to contain a null point, i.e. to have a non-empty intersection with $S(\Omega_{r,s})$ so that the
orthogonality condition will not become vacuous. (Otherwise, it could happen that there are no orthogonal points in $U$.) With this requirement,
Proposition \ref{orthShilov1} and Proposition  \ref{orthShilov} below will imply that orthogonal maps and holomorphic
maps between Shilov boundaries are the same set of mappings. However, we remark that in principle one can define orthogonal maps more generally without requiring $U\cap S(\Omega_{r,s})\neq\emptyset$ (but then we need to assume that $U$ contains at least a pair of orthogonal points) and some of our results (e.g. Theorem \ref{main}) will still be valid. Thus, we may say that a local orthogonal map is a slightly more general object.

\begin{definition}
An orthogonal map $ F: U\subseteq \mathcal{G}(r,s)  \to \mathcal{G}(r',s') $ is called \emph{null} if $ F(U) \subseteq S(\Omega_{r',s'}) $.
\end{definition}

From the definition of orthogonal maps and the properties of the Shilov boundary, we immediately derive the following result.

\begin{proposition}\label{orthShilov1}
Every local orthogonal map $ F: \mathcal{G}(r,s) \to \mathcal{G}(r',s') $ preserves null points.
\end{proposition}

From the definition, orthogonal maps preserves null points. The following proposition is essentially the converse.

\begin{proposition}\label{orthShilov}
Let $r,s>0$ and $F:U\subset \mathcal{G}(r,s)\to \mathcal{G}(r',s')$ be a holomorphic map, where $U$ is an open set containing a null point. If $F$ maps null points to null points, there exists an open set $U'\subset U $, such that $F:U'\to \mathcal{G}(r',s')$ is an orthogonal map.

\end{proposition}
\begin{proof}
By composing with automorphisms and shrinking $U$ if necessary, we may assume without loss of generality that $U$ is contained in the open set $U_1\in \mathcal{G}(r,s)$, where $U_1=\{\mathcal Z=[I_{r},Z]\in \mathbb C^{r\times {(r+s)}}\}$, $Z=(z_{ij})_{r\times s}\in \mathbb C^{r\times s}$ and $F(U)$ lies in the analogous chart $U_1'=\{\mathcal Z'=[I_{r'},Z']\in \mathbb C^{r'\times {(r'+s')}}\}$.

 In these coordinates, null points $\mathcal Z$ in $\mathcal{G}(r,s)$ are characterized by the equation $I-ZZ^H=0$. Explicitly, $\mathcal Z=[I_{r},Z]\in \mathcal{G}(r,s)$ is a null point if and only if 
 $$h_{ij}(Z,\bar Z ):=\sum_{k=1}^{s}z_{ik}\bar z_{kj}-\delta_{ij}=0,\quad (1\le i\le r, \quad 1\le j\le s)$$ 
 where $\delta_{ij}$ is the Kronecker function. 
 
 Now we can assume that $F=[I_{r'},F^{\sharp}]$ and $F^{\sharp}=(f_{ij})_{r'\times s'}$. Define
$$G_{ij}(Z,\bar Z):=\sum_{k=1}^{s'}f_{ik}(Z)\overline {f_{kj}(Z)}-\delta_{ij},\quad (1\le i\le r', \quad 1\le j\le s').$$

Here we use $f(Z)$ to denote $f(z_{11},\cdots,z_{rs})$.

 Since $F$ preserves null points, there exists a connected open set $U'\subset U$ containing a null point and some real analytic functions $g_{ij}^{kl}(Z,\bar Z)$ such that
$$G_{ij}(Z,\bar Z)=\sum_{k=1}^{s'}f_{ik}(Z)\overline {f_{kj}(Z)}-\delta_{ij}=\sum_{kl}h_{kl}(Z,\overline{Z})g_{ij}^{kl}( Z,\overline{Z}),$$ 
holds on $U'$. 

By polarization (after further shrinking $U'$ if necessary), this implies for all $\mathcal Z=[I_r,Z]$,$\mathcal W=[I_r,W]$ in $U'$, we have 
\begin{align}\label{bound eq}
G_{ij}(Z,\bar W)=\sum_{k=1}^{s'}f_{ik}(Z)\overline {f_{kj}(W)}-\delta_{ij}=\sum_{kl}h_{kl}(Z,\bar W)g_{ij}^{kl}(Z,\bar W).
\end{align}

For any point $\mathcal Z = [I_r, Z]\in \mathcal{G}(r,s)$, a point $\mathcal W= [I_r, W]$ lies in $\mathcal Z^{\perp}\cap U'$ precisely when $h_{kl}(Z,\bar{W}) = 0$ for all $k, l$. By \eqref{bound eq}, this implies $F(\mathcal W)\in F(\mathcal Z)^{\perp}$. Hence, $F$ preserves orthogonality locally, proving $F|_{U'}$ is an orthogonal map. 

\end{proof}

\section{Rigidity of local orthogonal maps between Grassmannians}

In this section, we will introduce some properties of orthogonal maps between Grassmannian. The following proposition is clear.

\begin{proposition}
If two points $p, q\in\mathcal{G}(r,s)$ satisfy $p\perp q$ and $\dim(V_p\cap V_q)>0$, then every vector $\alpha\in V_p\cap V_q$ is a null vector in $\mathbb{C}^{r, s}$.
\end{proposition}

Notation: Let $K\subset\mathcal{G}^{r, s}$. We use $V_{K}$ to denote the linear span of $\cup_{p\in K}V_{p}$ in $\mathbb{C}^{r, s}$.

\begin{lemma}\label{eq}
Let $s\ge 2$, $1\leq r'\le s'\in\mathbb{N}$, and let $F$ be a local orthogonal map from $U\subset\mathcal{G}(1,s)=\mathbb{P}^{1,s}$ to $\mathcal{G}(r',s')$. If
\[
\dim(V_{F(H\cap U)})=\dim(V_{F(U)})
\]
for every general hyperplane $H$ in $\mathbb{P}^{1,s}$, then $F$ is null.
\end{lemma}
\begin{proof}
For any general hyperplane $H\subset\mathbb{P}^{1,s}$, by definition, there exists a point $p\in \mathbb{P}^{1,s}$ such that $p\perp H$. 
Since the orthogonal map $F$ preserves orthogonality, for any vector $\alpha$ in $V_{F(p)}$, we have $\alpha\perp V_{F(q)}$
for any $q\in H$. This implies $\alpha\perp\beta$ for every $\beta\in V_{F(H)}$.

By the dimension condition $\dim(V_{F(H\cap U)})=\dim(V_{F(U)})$, the inclusion $V_{F(H\cap U)}\subset V_{F(U)}$ must be an equality. Hence, $\alpha\perp V_{F(U)}$. In particular, $\alpha\perp\alpha$, yielding $\langle\alpha,\alpha\rangle = 0$, so $\alpha$ is a null vector.

For any general $p\in U$, there exists a hyperplane $H\subset\mathbb{P}^{1,s}$ orthogonal to $p$. The argument above shows that all vectors $\alpha\subset V_{F(p)}$ are null. Thus, $V_{F(p)}$ is a null subspace of $\mathbb{C}^{r',s'}$, implying $F(p)\in S(\Omega_{r',s'})$. Since $p$ is arbitrary, $F(U)\subseteq S(\Omega_{r',s'})$, proving $F$ is null.

\end{proof}

The following proposition helps us convert the rigidity problems of orthogonal maps between Grassmannians to those of generalized balls.

\begin{proposition}\label{null}
Let $s \ge 2$, $r',s'\in\mathbb{N}$ with $2\le r' \le s'$, and let $F:U\subset\mathcal{G}(1,s)\to\mathcal{G}(r',s')$ be a local orthogonal map. If $s' - r' < 2(s - 1)$ and $F$ is non-null, there exists an $(r' - 1)$-dimensional null subspace $N\subset\mathbb{C}^{r',s'}$ such that $N\subset V_{F(p)}$ for every point $p\in\mathcal{G}(1,s)$.
\end{proposition}

\begin{proof}
For a general point $p\in\mathcal{G}(1,s)=\mathbb{P}^{1,s}$, denote $N_p = V_{F(p)}\cap V_{F(p^{\perp})}$ and $\dim N_p = k$. By orthogonality preservation, $V_{F(p)}\perp V_{F(p^{\perp})}$, implying $N_p\perp N_p$. Thus, $N_p$ is a null subspace in $\mathbb{C}^{r',s'}$. Hence $k\leq r'$.

Select $s + 1$ pairwise orthogonal non-null points $p_1,\ldots,p_{s + 1}\in\mathbb{P}^{1,s}$. For $i\neq j$, $V_{F(p_j)}\subset V_{F(p_i^{\perp})}$, yielding
\[
\dim\left(V_{F(p_i)}\cap\sum_{j\neq i}V_{F(p_j)}\right)\leq k.
\]

Since $\dim V_{F(p_i)} = r'$, there exist at least $(r' - k)$ linearly independent vectors $\alpha_1^i,\ldots,\alpha_{r' - k}^i\subset V_{F(p_i)}$ outside $\sum_{j\neq i}V_{F(p_j)}$. We verify the linear independence of the full set $\{\alpha_l^i\}_{1\leq l\leq r'-k}^{1\leq i \leq s+1}$.

For $\alpha^i=\sum_{l = 1}^{r' - k}x_l^i\alpha_l^i$ ($x_l^i\in\mathbb{C}$), linear independence of $\alpha^1,\ldots,\alpha^{s + 1}$ follows from $\alpha^i\notin\sum_{j\neq i}V_{F(p_j)}$. Hence,
\begin{equation}\label{1}
\dim\left(\sum_{i = 1}^{s + 1}V_{F(p_i)}\right)\geq(s + 1)(r' - k)+k.
\end{equation}

As $N_{p_i}\subset\mathbb{C}^{r',s'}$ is a $k$-dimensional null subspace, then $N_{p_i}^{\perp}$ is a $(r' -k,s'-k,k)$-subspace. For $j\geq2$, $V_{F(p_j)}\subset V_{F(p_1)}^{\perp}\subset N_{p_1}^{\perp}$ and $V_{F(p_1)}\perp N_{p_1}$, so
\begin{equation}\label{2}
\sum_{i = 1}^{s + 1}V_{F(p_i)}\subset N_{p_1}^{\perp}\quad\text{and}\quad\dim\left(\sum_{i = 1}^{s + 1}V_{F(p_i)}\right)\leq r'+s' - k.
\end{equation}

Combining inequalities (\ref{1}) and (\ref{2}) with $s' - r'<2(s - 1)$ and $s\geq 2$, we have
\[
(s + 1)(r' - k)+k\leq r'+s'-k < 2r'+2(s - 1)-k,
\]
yielding $k > r' - 2$. On the other hand, if $k = r'$, then $F$ must be null. By the assumption that $F$ is non-null, then $k = r' - 1$. So $V_{F(p)}$ contains an $(r' - 1)$-dimensional null subspace $N_p$ for all $p\in \mathcal{G}(1,s)\cong \mathbb{P}^{1,s}$.

To show $N_p = N_q$ for all $p$,$q\in \mathcal{G}(1,s)$, suppose otherwise.

For two points \( p \perp a \) in \( \mathbb{P}^{1,s} \), we have \( N_p \perp N_a \) in $\mathbb C^{r',s'}$. Suppose \( N_p \neq N_a \); since \( \dim N_p = \dim N_a = r' - 1 \), it follows that \( \dim(N_p + N_a) \geq r' \). This makes \( M_{pa} := N_p + N_a \) a maximal null subspace in \( \mathbb{C}^{r' + s'} \). For any point \( b \) satisfying \( b \perp p \) and \( b \perp a \), the orthogonality conditions \( N_b \perp N_a \) and \( N_b \perp N_p \) imply \(N_b \subset N_{pa} \). Consequently, we deduce that \( N_a + N_b = N_a + N_p = N_b + N_p = M_{pa} \).

Now we assume that $s=2$. From the argument above, we know $\dim (N_a + N_b)=r'$. Since $F$ is non-null, $\dim (V_{F(a)}+V_{F(b)}+V_{F(p)})\ge 3+r'$. On the other hand, because $V_{F(a)}$, $V_{F(b)}$ and $V_{F(p)}$ are all contained in $(N_a+N_b)^\perp$, then $\dim (V_{F(a)}+V_{F(b)}+V_{F(p)})\le s'$, which contradicts $s'-r'<2$.

When \( s \geq 3 \), for any points \( p, q \in \mathbb{P}^{1,s} \), there exist two points \( a, b \) such that \( a, b \in p^{\perp} \), \( a, b \in q^{\perp} \), and \( a \perp b \) By the preceding argument, \( N_p, N_q \subset M_{ab} \). It follows that either \( N_p = N_q \) or \( N_p + N_q =M_{ab}\) is a maximal null \((0,0,r')\)-subspace in \( \mathbb{C}^{r',s'} \). If $N_p\neq N_q$, then \( M_{pq} := N_p + N_q \) is a maximal null space.  Hence \( N_p \perp N_q \) for any two points \( p, q \). 

For any point \( c \), since \( K_{c} \) is an \((r' - 1)\)-dimensional null space in \( \mathbb{C}^{r',s'} \) and \( N_c \perp N_p, N_c \perp K_q \), it follows that \( N_{c} \subset M_{pq} \).

Now we have proved that there exists a maximal null space $M\subset \mathbb{C}^{r',s'}$ such that $N_p\subset N$ for any $p\in \mathbb{P}^{1,s}$ under the assumption that $N_p\neq N_q$ for any two points $p$, $q$.  Assume that $M^{\perp}=M\oplus K$ where $K$ is a $(0,s' - r',0)$-subspace for $M^{\perp}$ is a $(0,s' - r',r')$ subspace.

Since \( N_p = V_{F(p)} \cap V_{F(p^{\perp})} \) and \( N_p + N_q = M\) for any \( p, q \), it follows that \( V_{F(p)} \perp M \). For \( \dim V_{F(p)} = r' \) and \( \dim N_p = r' - 1 \), there exists a unique $1$-dimensional linear subspace $L_p$  such that \( V_{F(p)} = L_p \oplus N_p \) for any \( p \in \mathbb{P}^{1,s} \).

Let $F^{K}$ be a holomorphic map from $U\subset \mathbb{P}^{1,s}$ to $G(1,K)\cong\mathbb{P}^{0,s' - r'}$ defined by $F^{H}(p)=L_p$. The map $F^{H}$ is induced by $F$. Obviously,  $L_p\perp L_q$ when $p\perp q$. Hence $F^{H}$ is an orthogonal map from $\mathbb{P}^{1,s}$ to $\mathbb{P}^{0,s' - r'}$.

Theorem 3.3 in \cite{GN} shows that for any $r,s\in \mathbb{N}$, if $r'+s'\leq 2(r + s)-3$, then any orthogonal map from $\mathbb{P}^{r,s}$ to $\mathbb{P}^{r',s'}$ is either null or quasi-linear. On the other hand, it is clear that there is no local quasi-linear or null orthogonal map from $\mathbb{P}^{1,s}$ to $\mathbb{P}^{0,s' - r'}$, which is a contradiction.

So, $N_p=N_q$ for any two points $p$, $q$ in $\mathcal G(1,s)$. Hence, there exists an $(r' - 1)$-dimensional null subspace $N\subset\mathbb{C}^{r',s'}$ such that $N\subset V_{F(p)}$ for all $p\in\mathcal{G}(1,s)$.

\end{proof}

The following proposition comes from Proposition \ref{null} directly.
\begin{proposition}
Under the assumption of Proposition \ref{null}, there exist a $(0,0,r'-1)$-subspace $N$ and a $(1,s'-r'+1,0)$-subspace $K$ in $\mathbb C^{r',s'}$ such that $N\subset V_{F(p)}\subset K\oplus N=\mathbb C^{r',s'}$ for any $p\in U\subset \mathbb P^{1,s}$.
\end{proposition}

We can find the following rigidity theorem about orthogonal maps between projective spaces in \cite{GN}.	 

\begin{theorem}\label{ball}\cite{GN}
Let $l$,$l'\ge 2$ in $\mathbb N$ and $f$ be a local orthogonal map from $\mathbb P^{1,l}$ to $\mathbb P^{1,l'}$.

(1)If $l'< l$, then $f$ is a constant map.

(2)If $l\le l'\le 2l-2$, then $f$ is a linear map or a constant map.
\end{theorem}

Now we are ready to prove our main theorem.

\begin{theorem}\label{main11}
	Let $s\ge 2$, $2\le r'\le s'\in \mathbb N$ and $F$ be a local orthogonal map from $ \mathcal G(1,s)$ to $\mathcal G(r',s')$. 
	\begin{itemize}
	\item[(1)]If $s'-r'<s-1$, then $F$ is a constant map.
	
	\item[(2)] If $s-1\le s'-r'< 2s-2$, $F$ is a linear map.
After composing suitable automorphisms of $\mathcal G(1,s)$ and $\mathcal G(r',s')$,
		$F$ has the explicit form
		$$[z_0,\bf z]\to \left(\begin{array}{ccccc}
			I&0&I&0&0\\
			0&z_0&0&\bf z &0\\
		\end{array}\right),$$
		where $[z_1,{\bf z}]=[z_1,z_2,\cdots,z_{s+1}]$.

	\end{itemize}

\end{theorem}
\begin{proof}
By Proposition \ref{null}, there exists a $(0,0,r'-1)$-subspace $N$ and a $(1,s'-r'+1)$-subspace $K$ in $\mathbb C^{r',s'}$ with $K\perp N$ such that $N\subset V_{F(p)}\subset K\oplus N$ for any point $p\in U$. After composing with automorphism of $\mathbb C^{r',s'}$, we can assume that the representative matrix of $N$ is $[I_{r'-1},0,I_{r'-1},0]$.

Hence, we can assume that $F(p)=F_1(p)\oplus F_2(p)$, where $F_1:U\subset \mathcal{G}(1,s)\cong \mathbb P^{1,s}\to G(1,K)\cong \mathbb P^{1,s'-r'+1} $ and $F_2:\mathcal{G}(1,s)\cong \mathbb P^{1,s}\to G(r'-1,N)$. The Grassmannian variety $G(1,K)$ is the set of points bijectively corresponding to $1$-dimensional linear subspace in $K$.  Then $F_1$ is a local orthogonal map and $F_2$ is a constant map.

Therefore, we only need to character the orthogonal map $F_1$ from $\mathbb P^{1,s}$ to $\mathbb P^{1,s'-r'+1}$.

(1) When $s'-r'<s-1$, $F_1$ is a constant map by Theorem \ref{ball}. Hence, $F$ is a constant map.

(2)When $s-1\le s'-r'< 2s-2$, $F_1$ is a linear embedding from $\mathbb P^{1,s}$ to $\mathbb P^{1,s'-r'+1}$ by Theorem \ref{ball}.
Hence, $F$ is a linear map. In addition, using the orthogonal representative matrix and composing the suitable automorphisms of both sides, $F$ can be defined by
$$[z_0,\bf z]\to \left(\begin{array}{ccccc}
			I_{r'-1}&0&I_{r'-1}&0&0\\
			0&z_0&0&\bf z &0\\
		\end{array}\right),$$
		where $[z_0,{\bf z}]=[z_0,z_1,\cdots,z_{s}]$.

\end{proof}

From the relationship between the orthogonal map and the holomorphic map between Type I bounded symmetric domains  preserving Shilov boundaries, we can get the following result directly.

\begin{corollary}\label{main2}
		Let $s\ge 2$, $2\le r'\le s'\in \mathbb N$, and let $U\subset \mathbb C^{s}$ be an open neighborhood of a point in  $ S(\Omega_{1,s})$.
		let $F$ be a holomorphic  map from $U$ to $\mathbb C^{r's'}$ satisfying 
		$F(U\cap S(\Omega_{1,s}) )\subset S(\Omega_{r',s'})$.
	\begin{itemize}
	\item[(1)]If $s'-r'<s-1$, then $F$ is a constant map.
	
	\item[(2)] If $s-1\le s'-r'< 2s-2$, $F$ is a linear map. After composing with suitable automorphisms of $ \Omega_{1,s}$ and $ \Omega_{r',s'}$,
		$F$ is given by the following
		$$z\to \left(\begin{array}{ccc}
			I&0&0\\
			0& z &0\\
		\end{array}\right).$$

	\end{itemize}

\end{corollary}

{\bf Example :} The following generalized Whitney map $f$ from $\Omega_{1,s}$ to $\Omega_{r',r'+2s-1}$ is given by

$$z=[z_1,\cdots,z_s] \mapsto \left(\begin{array}{ccccccccccccc}
		 z_1&\cdots&z_{s-1}&z_1z_s&z_2z_s&\cdots&z_{s}^2
		&0&\cdots&0\\ 	
		0&\cdots&0&0&0&\cdots&0
				&1&\cdots&0\\
				
				 \vdots&\ddots&\vdots&\vdots&\vdots&\ddots&\vdots
								&\vdots&\ddots&\vdots\\
								
								 0&\cdots&0&0&0&\cdots&0
												&0&\cdots&1\\
				
	\end{array}\right)$$
	preserving the Shilov boundaries . This example shows  the inequalities in Theorem \ref{main11} and Corollary \ref{main2} governing the rigidity are sharp.
	
{\bf Acknowledgements}
The author is deeply grateful to Prof. Sui-Chung Ng for the invaluable discussions, constructive suggestions and detailed feedback throughout the development of this paper. She also wishes to thank Prof. Zaitsev whose comments not only has led to better phrasing of Corollary 1.3 but also helped refining other parts of the work. Additionally, thanks are due to Prof. Xiaojun Huang and Prof. Chenyang Xu for their insightful comments and valuable suggestions. This research was supported by the National Natural Science Foundation of China (NSFC) under Grants No. 12471042 and 12471078.

\end{document}